\documentclass[10pt]{amsart}
\usepackage{amssymb}
\usepackage{bm}
\usepackage{graphicx}
\usepackage[centertags]{amsmath}
\usepackage{amsfonts}
\usepackage{amsthm}
\usepackage{graphicx}
\newtheorem{thm}{Theorem}
\newtheorem{cor}[thm]{Corollary}
\newtheorem{lem}[thm]{Lemma}

\newtheorem{prop}[thm]{Proposition}

\newtheorem{defn}[thm]{Definition}
\theoremstyle{definition}

\newcommand{\nn}{\mathbb{N}}
\newcommand{\ee}{\varepsilon}
\newcommand{\con}{\smallfrown}
\newcommand{\meg}{\geqslant}
\newcommand{\mik}{\leqslant}
\newcommand{\lines}{\mathrm{Lines}}
\newcommand{\dens}{\mathrm{dens}}
\newcommand{\gr}{\mathrm{GR}}
\newcommand{\ave}{\mathbb{E}}
\newcommand{\dhj}{\mathrm{DHJ}}
\newcommand{\mdhj}{\mathrm{MDHJ}}

\begin{document}

\title{A simple proof of the density Hales--Jewett theorem}

\author{Pandelis Dodos, Vassilis Kanellopoulos and Konstantinos Tyros}

\address{Department of Mathematics, University of Athens, Panepistimiopolis 157 84, Athens, Greece}
\email{pdodos@math.uoa.gr}

\address{National Technical University of Athens, Faculty of Applied Sciences,
Department of Mathematics, Zografou Campus, 157 80, Athens, Greece}
\email{bkanel@math.ntua.gr}

\address{Department of Mathematics, University of Toronto, Toronto, Canada M5S 2E4}
\email{ktyros@math.toronto.edu}


\maketitle


\begin{abstract}
We give a purely combinatorial proof of the density Hales--Jewett Theorem that is modeled after Polymath's proof but is significantly
simpler. In particular, we avoid the use of the equal-slices measure and work exclusively with the uniform measure.
\end{abstract}


\section{Introduction}

We begin by introducing some pieces of notation and some terminology. For every pair $k,n$ of positive integers let $[k]^n$ be the set
of all sequences of length $n$ having values in $[k]:=\{1,...,k\}$. The elements of $[k]^n$ will be referred to as \textit{words}.
Also fix a letter $v$. A \textit{variable word} is a finite sequence of length $n$ having values in $[k]\cup\{v\}$ where the letter $v$
appears at least once. If $\ell$ is a variable word and $i\in[k]$, then $\ell(i)$ is the word obtained by substituting all appearances
of the letter $v$ in $\ell$ by $i$. A \textit{combinatorial line} of $[k]^n$ is a set of the form $\{\ell(i):i\in[k]\}$ where $\ell$ is
a variable word. If $A$ is a subset of $[k]^n$, then its \textit{density} is the quantity $|A|/k^n$ where $|A|$ stands for the cardinality
of the set $A$.

The following result is known as the \textit{density Hales--Jewett Theorem} and is due to H. Furstenberg and Y. Katznelson \cite{FK2}.
\begin{thm} \label{t1}
For every integer $k\meg 2$ and every $0<\delta\mik 1$ there exists an integer $N$ with the following property. If $n\meg N$ and
$A$ is a subset of $[k]^n$ of density $\delta$, then $A$ contains a combinatorial line of $[k]^n$. The least integer $N$ with this
property will be denoted by $\dhj(k,\delta)$.
\end{thm}
The density Hales--Jewett Theorem is a fundamental result of Ramsey Theory. It has several strong results as consequences, most notably
the famous Szemer\'{e}di Theorem on arithmetic progressions \cite{Sz} and its multidimensional version \cite{FK1}.

Because of its significance the density Hales--Jewett Theorem has received considerable attention and there are, by now, several different
proofs \cite{Au,Pol,Tao}.  Our goal in this paper is to give yet another proof of the density Hales--Jewett Theorem that is modeled after
Polymath's proof \cite{Pol} but places one of its crucial parts in a general conceptual framework. In fact, the argument was found in the
course of obtaining a density version of the Carlson--Simpson Theorem \cite{DKT} and we decided to present it also within the context of
the density Hales--Jewett Theorem since it simplifies the method in \cite{Pol}.

To proceed with our discussion it is useful at this point to recall the strategy of Polymath's proof. It is based on the density
increment method. Specifically, one argues that if a subset $A$ of $[k]^n$ of density $\delta$ fails to contain a combinatorial line,
then $A$ has density $\delta+\gamma$ inside a large subspace of $[k]^n$ where $\gamma$ is a positive constant that depends only on $\delta$;
once this is done Theorem \ref{t1} follows by a standard iteration. To achieve this goal, one proceeds in two steps: firstly one shows that $A$
must correlate with a ``structured" set $B$ more than expected, and then argues that the ``structured" set $B$ can be partitioned in subspaces.

The proof of the second step given in \cite{Pol} is a non-trivial modification of an argument due to M. Ajtai and E. Szemer\'{e}di \cite{AS}.
It is essentially a ``greedy" algorithm with an elegant proof that appears to be optimal, and we offer no new insight.

To execute the first step it is necessary to have a ``probabilistic" version of Theorem \ref{t1}. This means that a dense subset of $[k]^n$
not only will contain a combinatorial line but, actually, a non-trivial portion of them. Unfortunately, such a naive ``probabilistic" version
is false. To overcome this problem the participants of the polymath project introduced the \textit{equal-slices} measure, a probability
measure on $[k]^n$, and argued that for the equal-slices measure Theorem \ref{t1} does have a density version.  While the idea of changing
the measure is an important one, it necessitates a number of tools whose relevance to Theorem \ref{t1} can be justified only a posteriori.

We propose a different way to obtain such a ``probabilistic" version that enables us to work exclusively with the uniform measure on $[k]^n$.
Our approach is based on an old paper of P. Erd\H{o}s and A. Hajnal \cite{EH} that initiated the study of the following general problem in
Ramsey Theory. Suppose that we are given a Ramsey space $\mathbb{S}$; for concreteness the reader may think of $[k]^n$ for some large $n$.
Suppose, further, that we are given a family $\{A_s:s\in\mathbb{S}\}$ of measurable events in a probability space $(\Omega,\Sigma,\mu)$
satisfying $\mu(A_s)\meg \delta>0$ for every $s\in\mathbb{S}$. The goal is then to find a ``substructure" $\mathbb{S}'$ of $\mathbb{S}$
(in the case of the density Hales--Jewett Theorem, $\mathbb{S}'$ is a combinatorial line of $[k]^n$) such that the events in the family $\{A_s:s\in\mathbb{S}'\}$ are highly correlated. Many density results in Ramsey Theory can be formulated in this way and so does the
density Hales--Jewett Theorem; see, \cite[Proposition 2.1]{FK2}. It is precisely this form that we are taking advantage of, together
with some simple coloristic and averaging arguments, and execute the first step.

Some final remarks about how this paper is written. We have made no attempt to optimize the argument. Instead, we tried to make the exposition
as clear as possible. The bounds we get have an Ackermann-type dependence with respect to $k$ and coincide, essentially, with the bounds from
Polymath's proof for all sufficiently large values of $k$. The fundamental problem whether there exist primitive recursive bounds for the
numbers $\dhj(k,\delta)$ is open and is likely to require a more sophisticated approach.


\section{Background material}

By $\nn=\{0,1,2,...\}$ we denote the natural numbers. As we have already mentioned, the cardinality of a set $X$ will be denoted by $|X|$.
For every nonempty finite set $X$ by $\ave_{x\in X}$ we shall denote the average $\frac{1}{|X|} \sum_{x\in X}$. If it is clear which set
$X$ we are referring to, then this average will be denoted simply by $\ave_x$.

We recall some definitions related to the Hales--Jewett Theorem \cite{HJ}. Specifically, let $k,m,n\in\nn$ with $k\meg 2$ and $n\meg m\meg1$
and fix an $m$-tuple $v_1,...,v_m$ of distinct letters. An \textit{$m$-variable word} of $[k]^n$ is a finite sequence of length $n$ having
values in $[k]\cup\{v_1,...,v_m\}$ where, for each $j\in [m]$, the letter $v_j$ appears at least once. If $z$ is an $m$-variable word and
$a_1,...,a_m\in[k]$, then $z(a_1,...,a_m)$ is the word obtained by substituting in $z$ the letter $v_j$ with $a_j$ for every $j\in [m]$.
An \textit{$m$-dimensional subspace} of $[k]^n$ is a set of the form $\{z(a_1,...,a_m):a_1,...,a_m\in[k]\}$ where $z$ is an $m$-variable word.
Observe that an $1$-dimensional subspace of $[k]^n$ is just a combinatorial line. If $V$ is an $m$-dimensional subspace of $[k]^n$, then by
$\lines(V)$ we shall denote the set of all combinatorial lines of $[k]^n$ that are contained in $V$. Moreover, for every subset $A$ of $[k]^n$
the \textit{density of $A$ in $V$}, denoted by $\dens_V(A)$, is the quantity $|A\cap V|/|V|$. The density of $A$ in $[k]^n$ will be denoted
simply by $\dens(A)$.

Let $V$ be an $m$-dimensional subspace of $[k]^n$ and $z$ be the $m$-variable word that generates it. Notice that $z$ induces a natural
``isomorphism" between $[k]^m$ and $V$ defined by $[k]^m\ni (a_1,...,a_m)\mapsto z(a_1,...,a_m)\in V$. Thus, in practice, we may identify
$m$-dimensional subspaces of $[k]^n$ with ``copies" of $[k]^m$ inside $[k]^n$. Having this identification in mind, for every $k'\in\nn$
with $2\mik k'\mik k$ we set
\[V\upharpoonright k'=\{ z(a_1,...,a_m): a_1, ..., a_m\in [k']\}.\]

Now let $n, l\in\nn$ with $n, l\meg 1$. For every $x\in [k]^n$ and every $y\in [k]^l$ by $x^{\con}y$ we shall denote the concatenation
of $x$ and $y$. Notice that $x^{\con} y\in [k]^{n+l}$. More generally, if $A\subseteq [k]^n$ and $B\subseteq [k]^l$ then we set
$A^{\con} B=\{x^{\con} y: x\in A \text{ and } y\in B\}$.

Finally we record, for future use, the following consequence of the Graham--Rothschild Theorem \cite{GR}.
\begin{prop} \label{p2}
For every integer $k\meg 2$ and every integer $m\meg 1$ there exists an integer $N$ with the following property. For every integer $n\meg N$
and every set $\mathcal{L}$ of combinatorial lines of $[k]^n$ there exists an $m$-dimensional subspace $V$ of $[k]^n$ such that either $\lines(V)\subseteq \mathcal{L}$ or $\lines(V)\cap\mathcal{L}=\varnothing$. The least integer $N$ with this property will be denoted by $\gr(k,m)$.
\end{prop}
Proposition \ref{p2} can be proved by repeated applications of the Hales--Jewett Theorem much in the spirit of Ramsey's classical Theorem;
see, e.g., \cite[Theorem 2.4.1]{McC}. Another excellent and short proof can be found in \cite[\S 4]{PV}. Also we notice that there exist
reasonable upper bounds for the numbers $\gr(k,m)$. Specifically, it follows from the work of S. Shelah \cite{Sh} that there exists a primitive
recursive function $\phi\colon\nn^2\to\nn$ belonging to the class $\mathcal{E}^{6}$ of Grzegorczyk's hierarchy such that for every integer
$k\meg 2$ and every integer $m\meg 1$ we have $\gr(k,m)\mik \phi(k,m)$.


\section{Preliminary tools}

In this section we will gather some preliminary tools which are needed for the proof of Theorem \ref{t1} but are not directly related to the
main argument. To simplify the exposition, below and in the rest of the paper, we will write ``$\dhj_k$" to denote the proposition that for
every  $0<\delta\mik 1$ the number $\dhj(k,\delta)$ is finite.

The first result, taken from \cite{FK2}, asserts that the density Hales--Jewett Theorem implies its multidimensional version.
\begin{prop} \label{p3}
Let $k\in\nn$ with $k\meg 2$ and assume $\dhj_k$. Then for every integer $m\meg 1$ and every $0<\delta\mik 1$
there exists an integer $\mdhj(k,m,\delta)$ with the following property. If $n\meg \mdhj(k,m,\delta)$, then every subset $A$ of $[k]^n$
of density at least $\delta$ contains an $m$-dimensional subspace of $[k]^n$.
\end{prop}
\begin{proof}
By induction on $m$. The case ``$m=1$" is the content of $\dhj_k$. Let $m\in\nn$ with $m\meg 1$ and assume that the result has been proved
up to $m$. For every $0<\delta\mik 1$ we set $\mdhj(k,m+1,\delta)=M+\mdhj( k,m, \delta 2^{-1} (k+1)^{-M})$ where $M=\dhj(k,\delta/2)$.
We claim that with this choice the result follows. Indeed, let $n\meg \mdhj(k,m+1,\delta)$ and fix a subset $A$ of $[k]^n$ with
$\dens(A)\meg\delta$. For every $x\in [k]^{n-M}$ let $A_x=\{y\in [k]^M: x^{\con}y\in A\}$. Notice that $\ave_x \dens(A_x)\meg\delta$.
Therefore, there exists a subset $B$ of $[k]^{n-M}$ with $\dens(B)\meg \delta/2$ such that for every $x\in B$ we have $\dens(A_x)\meg\delta/2$.
By the choice of $M$, for every $x\in B$ there exists a combinatorial line $\ell_x$ of $[k]^M$ such that $\ell_x\subseteq A_x$. The number
of combinatorial lines of $[k]^M$ is less than $(k+1)^M$. Therefore, there exist a combinatorial line $\ell$ of $[k]^M$ and a subset $C$
of $B$ with $\dens(C)\meg \delta2^{-1}(k+1)^{-M}$ such that $\ell\subseteq A_x$ for every $x\in C$. Since $n-M\meg \mdhj( k,m, \delta 2^{-1}
(k+1)^{-M})$ there exists an $m$-dimensional subspace $W$ of $[k]^{n-M}$ with $W\subseteq C$. We set $V=W^{\con}\ell$. Then $V$ is an
$(m+1)$-dimensional subspace of $[k]^n$ and clearly $V\subseteq A$. The proof is completed.
\end{proof}
The second result asserts that every dense subset of $[k]^n$ becomes extremely uniformly distributed when restricted to a suitable
subspace of $[k]^n$.
\begin{lem} \label{l4}
Let $k,m\in\nn$ with $k\meg 2$ and $m\meg 1$. Also let $0<\ee<1$. If $n\meg \ee^{-1}k^mm$, then for every subset $A$ of $[k]^n$ with $\dens(A)>\ee$
there exist some $l<n$ and an $m$-dimensional subspace $V$ of $[k]^l$ such that for every $x\in V$ we have $\dens(A_x)\meg \dens(A)-\ee$ where
$A_x=\{y\in [k]^{n-l}: x^{\con} y\in A\}$.
\end{lem}
\begin{proof}
We set $V_1=[k]^m$ and we observe that $\ave_{x\in V_1} \dens(A_x)=\dens(A)$. Also let $\varrho=\ee(k^m-1)^{-1}$. If $V_1$ does not satisfy the
requirements of the lemma, then there exists $x_1\in V_1$ such that $\dens(A_{x_1})\meg \dens(A)+\varrho$. Next we set $V_2=x_1^{\con}[k]^m$ and
we notice that $\ave_{x\in V_2} \dens(A_x) \meg\dens(A)+\varrho$. If $V_2$ does not satisfy the requirements of the lemma, then there exists
$x_2\in V_2$ such that $\dens(A_{x_2})\meg \dens(A)+2\varrho$. This process must, of course, terminate after at most $\lfloor\varrho^{-1}\rfloor$ iterations. Noticing that $(\lfloor\varrho^{-1}\rfloor+1)m<n$ the result follows.
\end{proof}
Combining Proposition \ref{p3} and Lemma \ref{l4} we get the following corollary.
\begin{cor} \label{c5}
Let $k\in\nn$ with $k\meg 2$ and assume $\dhj_k$. Then for every integer $m\meg 1$ and every $0<\delta\mik 1$ there exists an integer
$\mdhj^*(k,m,\delta)$ with the following property. If $n\meg \mdhj^*(k,m,\delta)$, then for every subset $A$ of $[k+1]^n$ of density
at least $\delta$ there exists an $m$-dimensional subspace $V$ of $[k+1]^n$ such that $V\upharpoonright k$ is contained in $A$.
\end{cor}
\begin{proof}
We set $\mdhj^*(k,m,\delta)=(\delta/2)^{-1}(k+1)^M M$ where $M=\mdhj(k,m,\delta/2)$. Let $n\meg \mdhj^*(k,m,\delta)$ and fix a subset
$A$ of $[k+1]^n$ with $\dens(A)\meg \delta$. By Lemma \ref{l4}, there exist some $l<n$ and an $M$-dimensional subspace $W$ of $[k+1]^l$
such that $\dens(A_x)\meg \delta/2$ for every $x\in W$. We set $Z=W\upharpoonright k$. On the one hand, we have $|A\cap (Z^{\con} [k+1]^{n-l})|
\meg (\delta/2) |Z^{\con}[k+1]^{n-l}|$ since $\dens(A_x)\meg\delta/2$ for every $x\in Z$. On the other hand, the family $\{Z^{\con}y:
y\in [k+1]^{n-l}\}$ forms a partition of $Z^{\con}[k+1]^{n-l}$ into sets of equal size. Hence, there exists $y_0\in [k+1]^{n-l}$ such that
$|A\cap (Z^{\con} y_0)|\meg (\delta/2) |Z^{\con}y_0|$. Observe that $Z^{\con}y_0$ is isomorphic to $[k]^M$. Thus, by the choice of $M$,
there exists an $m$-dimensional subspace $\tilde{V}$ of $Z^{\con}y_0$ such that $\tilde{V}\subseteq A$. Let $V$ be the unique $m$-dimensional
subspace of $[k+1]^n$ with $V\upharpoonright k=\tilde{V}$. Then $V$ is as desired.
\end{proof}


\section{Proof of Theorem \ref{t1}}

The proof proceeds by induction on $k$. The case ``$k=2$" follows from the classical Sperner Theorem \cite{Sp}. So let $k\in\nn$ with $k\meg 2$
and assume $\dhj_k$. First we introduce some numerical invariants. Specifically, for every $0<\delta\mik 1$ we set
\begin{equation} \label{e1}
m_0=\dhj(k,\delta/4), \ \ \theta=\frac{\delta/4}{(k+1)^{m_0}-k^{m_0}}, \ \ \eta=\frac{\delta\theta}{48} \ \ \text{ and } \
\gamma=\frac{\delta\eta^2}{k}.
\end{equation}
The main step of the proof of $\dhj_{k+1}$ is the following dichotomy.
\begin{prop} \label{p6}
Let $k\in\nn$ with $k\meg 2$ and assume $\dhj_k$. Then for every $0<\delta\mik 1$ and every integer $d\meg 1$ there exists an integer $N(k,d,\delta)$
with the following property. If $n\meg N(k,d,\delta)$ and $A$ is a subset of $[k+1]^n$ with $\dens(A)\meg \delta$, then either $A$ contains a
combinatorial line of $[k+1]^n$, or there exists a $d$-dimensional subspace $V$ of $[k+1]^n$ such that $\dens_V(A)\meg \delta+\gamma/2$ where
$\gamma$ is as in \eqref{e1}.
\end{prop}
Using Proposition \ref{p6} the numbers $\dhj(k+1,\delta)$ can be estimated easily via a standard iteration. And, of course, this is enough to
complete the proof of the density Hales--Jewett Theorem.

It remains to prove Proposition \ref{p6}. This is our goal in the following subsection.

\subsection{Proof of Proposition \ref{p6}}

The proof is based on a series of lemmas. \textit{We emphasize that, in what follows, we will assume $\dhj_k$}. Also, for every integer
$m\meg 1$ and every $0<\ee\mik 1$ we set
\begin{equation} \label{e2}
n(m,\ee)=\ee^{-1}(k+1)^m m.
\end{equation}
We start with the following lemma.
\begin{lem}\label{l7}
Let $0<\delta\mik 1$ and  $m\in\nn$ with $m\meg m_0$. If $n\meg n(\mathrm{GR}(k,m),\eta^2/2)$, then for every subset $A$ of $[k+1]^n$
with $\mathrm{dens}(A)\meg \delta$ there exist some $l<n$ and an $m$-dimensional subspace $U$ of $[k+1]^{l}$ such that
\begin{enumerate}
\item[(a)] for every $u\in U$ we have $\mathrm{dens}(A_u)\meg\delta-\eta^2/2$, and
\item[(b)] for every $\ell\in \mathrm{Lines}(U\upharpoonright k)$ we have $\mathrm{dens}\big(\bigcap_{u\in\ell} A_u\big)\meg \theta$,
\end{enumerate}
where, as in Lemma \ref{l4}, $A_u=\{y\in [k+1]^{n-l}:u^{\con}y\in A\}$ for every $u\in U$.
\end{lem}
\begin{proof}
We apply Lemma \ref{l4} and we get some $l<n$ and a subspace $V$ of $[k+1]^l$ of dimension $\mathrm{GR}(k,m)$ such that
$\mathrm{dens}(A_v)\meg \delta-\eta^2/2$ for every $v\in V$. We set
\[\mathcal{L}=\Big\{ \ell\in\mathrm{Lines}(V\upharpoonright k): \mathrm {dens}\Big(\bigcap_{v\in\ell} A_v\Big)\meg \theta\Big\}.\]
By Proposition \ref{p2}, there exists an $m$-dimensional subspace $Y$ of $V\upharpoonright k$ such that either
$\mathrm{Lines}(Y)\subseteq \mathcal{L}$ or $\mathrm{Lines}(Y)\cap \mathcal{L}=\varnothing$. If $\mathrm{Lines}(Y)\subseteq \mathcal{L}$,
then let $U$ be the unique subspace of $[k+1]^l$ such that $U\upharpoonright k=Y$. It is easily checked that $U$ satisfies the requirements
of the lemma.

Therefore the proof will be completed once we show that $\mathrm{Lines}(Y)\cap\mathcal{L}\neq \varnothing$. To this end, first, we select
an $m_0$-dimensional subspace $Z$ of $Y$. By the choice of $\eta$ in (\ref{e1}) and the fact that $Z\subseteq V$, we have $\mathrm{dens}(A_z)\meg\delta/2$ for every $z\in Z$. Hence there exists $B\subseteq [k+1]^{n-l}$ with $\dens(B)\meg\delta/4$
such that $|A\cap (Z^{\con}y)|\meg (\delta/4)|Z^{\con}y|$ for every $y\in B$. Let $y\in B$ be arbitrary. By the previous discussion
and the choice of $m_0$ in (\ref{e1}), there exists $\ell_y\in\lines(Z)$ such that $\ell_y^{\con}y\subseteq A$. The number of combinatorial
lines of $Z$ is $(k+1)^{m_0}-k^{m_0}$. It follows that there exist $\ell_0\in\lines(Z)$ and a subset $C$ of $B$ with $\dens(C)\meg \theta$
such that $\ell_0^{\con}y\subseteq A$ for every $y\in C$. This implies that $\ell_0\in \lines(Y)\cap \mathcal{L}$ and the proof is completed.
\end{proof}
The next result asserts that if we have ``lack of density increment", then we can find a subspace $W$ of $[k+1]^n$ of sufficiently large
dimension satisfying two properties. Firstly the density of $A$ inside $W$ is essentially the same as the density of $A$ in $[k+1]^n$
and, secondly, with plenty of lines contained in $A\cap (W\upharpoonright k)$.
\begin{lem}\label{l8}
Let $0<\delta\mik 1$ and $m\in\nn$ with $m\meg m_0$. Also let $n\meg n(\mathrm{GR}(k,m),\eta^2/2)$ and $A$ be a subset of $[k+1]^n$
with $\mathrm{dens}(A)\meg \delta$. Then either there exists an $m$-dimensional subspace $X$ of $[k+1]^n$ such that $\mathrm{dens}_X(A)\meg
\delta+\eta^2/2$, or there exists an $m$-dimensional subspace $W$ of $[k+1]^n$ such that $\mathrm{dens}_W(A)\meg \delta-2\eta$ and
\begin{equation} \label{e3}
|\{\ell\in\lines(W\upharpoonright k): \ell\subseteq A\}|\meg (\theta/2) |\lines(W\upharpoonright k)|.
\end{equation}
\end{lem}
\begin{proof}
Clearly we may assume that $\dens_X(A)< \delta+\eta^2/2$ for every $m$-dimensional subspace $X$ of $[k+1]^n$. By Lemma \ref{l7}, there
exist some $l<n$ and an $m$-dimensional subspace $U$ of $[k+1]^l$ such that $\mathrm{dens}(A_u)\meg \delta-\eta^2/2$ for every $u\in U$
and $\mathrm{dens}\big(\bigcap_{u\in \ell} A_u\big)\meg \theta$ for every $\ell\in \lines(U\upharpoonright k)$.

The first property implies, in particular, that $\mathbb{E}_{y\in [k+1]^{n-l}}\mathrm{dens}_{U^\con y}(A)\meg \delta-\eta^2/2$.
For every $y\in [k+1]^{n-l}$ the set $U^\con y$ is an $m$-dimensional subspace of $[k+1]^n$. Thus, by our assumptions, we have
$\mathrm{dens}_{U^\con y}(A)< \delta+\eta^2/2$ for every $y\in [k+1]^{n-l}$. It follows that there exists a subset $H_1$ of $[k+1]^{n-l}$
with $\dens(H_1)\meg 1-\eta$ such that $\mathrm{dens}_{U^\con y}(A)\meg \delta-2\eta$ for every $y\in H_1$.

Now for every $y\in [k+1]^{n-l}$ let $\mathcal{L}_y=\{\ell\in\lines(U \upharpoonright k): y\in \bigcap_{u\in \ell} A_u\}$.
Since $\mathrm{dens}\big(\bigcap_{u\in \ell} A_u\big)\meg \theta$ for every $\ell\in \lines(U\upharpoonright k)$ we have
\[ \ave_{y\in [k+1]^{n-l}} \frac{|\mathcal{L}_y|}{|\lines(U\upharpoonright k)|} =
\ave_{\ell\in\lines(U\upharpoonright k)} \dens\Big( \bigcap_{u\in \ell} A_u\Big) \meg \theta. \]
Hence, there exists a subset $H_2$ of $[k+1]^{n-l}$ with $\dens(H_2)\meg\theta/2$ such that $|\mathcal{L}_y|\meg (\theta/2)|
\lines(U\upharpoonright k)|$ for every $y\in H_2$.

By the choice of $\theta$ and $\eta$ in (\ref{e1}), we have $\eta<\theta/2$. It follows that the set $H_1\cap H_2$ is nonempty.
We select $y_0\in H_1\cap H_2$ and we set $W=U^\con y_0$. It is easy to check that $W$ is as desired.
\end{proof}
From this point on the proof follows the steps of Polymath's proof. A crucial ingredient (perhaps the single most important one)
is the notion of an insensitive set which we are about to recall. To this end, we will need the following terminology.
Let $x,y\in [k+1]^n$ and write $x=(x_r)_{r=1}^n$ and $y=(y_r)_{r=1}^n$. Also let $i,j\in [k+1]$ with $i\neq j$. We say that $x$ and $y$
are \textit{$(i,j)$-equivalent} if for every $s\in [k+1]\setminus \{i,j\}$ we have $\{r\in [n]: x_r=s\}=\{r\in [n]: y_r=s\}$.
\begin{defn} \label{d9}
Let $i,j\in [k+1]$ with $i\neq j$ and $A$ be a subset of $[k+1]^n$. The set $A$ is said to be \emph{$(i,j)$-insensitive} provided that
for every $x\in A$ and every $y\in [k+1]^n$ if $x$ and $y$ are $(i,j)$-equivalent, then $y\in A$.

If $V$ is an $m$-dimensional subspace of $[k+1]^n$ and $A$ is a subset of $V$, then $A$ is said to be \emph{$(i,j)$-insensitive in} $V$
if, identifying $V$ with $[k+1]^m$, $A$ becomes an $(i,j)$-insensitive subset of $[k+1]^m$.
\end{defn}
It is easy to see that the family of all $(i,j)$-insensitive subsets of $[k+1]^n$ is closed under intersections, unions and complements.
The same remark, of course, applies to the family of all $(i,j)$-insensitive sets of a subspace $V$ of $[k+1]^n$.

Also we need to introduce some more numerical invariants. Precisely, for every $0<\delta\mik 1$ let $m_0$ and $\eta$ be as in (\ref{e1})
and set
\begin{equation} \label{e4}
\lambda=\frac{k+1}{k} \ \ \text{ and } \ \ M_0=\max\Big\{m_0, \frac{\log \eta^{-1}}{\log\lambda}\Big\}.
\end{equation}
We proceed with the following lemma.
\begin{lem} \label{l10}
Let $0<\delta\mik 1$ and $m\in\nn$ with  $m\meg M_0$. Let $n\meg n(\mathrm{GR}(k,m),\eta^2/2)$ and $A$ be a subset of $[k+1]^n$
with $\dens(A)\meg \delta$. Assume that $A$ contains no combinatorial line of $[k+1]^n$ and $\dens_X(A)< \delta+\eta^2/2$ for every
$m$-dimensional subspace $X$ of $[k+1]^n$. Then there exist an $m$-dimensional subspace $W$ of $[k+1]^n$ and a subset $C$ of $W$
satisfying the following properties.
\begin{enumerate}
\item [(a)] We have $\dens_W(C)\meg\theta/4$ and $C=\bigcap_{i=1}^k C_i$ where $C_i$ is $(i,k+1)$-insensitive in $W$ for every
$i\in [k]$.
\item [(b)] We have $\dens_W\big(A\cap(W\setminus C)\big) \meg (\delta+6\eta) \dens_W(W\setminus C)$ and, moreover,
$\mathrm{dens}_W\big(A\cap(W\setminus C)\big)\meg \delta-3\eta$.
\end{enumerate}
\end{lem}
\begin{proof}
By our assumptions, we may apply Lemma \ref{l8} and we get an $m$-dimensional subspace $W$ of $[k+1]^n$ such that $\dens_W(A)\meg\delta-2\eta$
and satisfying inequality (\ref{e3}). For every $\ell\in\lines(W\upharpoonright k)$ let $\bar{\ell}$ be the unique combinatorial line of $W$
such that $\bar{\ell}\upharpoonright k=\ell$. Let $B=\{\bar{\ell}(k+1): \ell\in\lines(W\upharpoonright k) \text{ with } \ell\subseteq A\}$ and
set  $C=B\cup \big(A\cap (W\upharpoonright k)\big)$. We will show that $W$ and $C$ are as desired. First we argue for (a). Identifying $W$ with
$[k+1]^m$, for every $x\in W$ let $x^{{k+1}\to i}$ be the unique element of $W$ obtained by replacing all appearances of $k+1$ in $x$ by $i$.
Setting $C_i=\{x\in W: x^{{k+1}\to i}\in A\cap (W\upharpoonright k)\}$ for every $i\in [k]$, we see that $C_i$ is $(i,k+1)$-insensitive in $W$
and $C=C_1\cap ... \cap C_k$. Next observe that the map $\lines(W\upharpoonright k)\ni \ell\mapsto \bar{\ell}(k+1)\in W$ is one-to-one. Hence,
\begin{eqnarray*}
|C| & \meg & |B| = |\{\ell\in\lines(W\upharpoonright k):\ell\subseteq A\}| \stackrel{(\ref{e3})}{\meg} (\theta/2)|\lines(W\upharpoonright k)| \\
& = & (\theta/2)((k+1)^m-k^m) \stackrel{(\ref{e4})}{\meg}(\theta(1-\eta)/2)(k+1)^m \stackrel{(\ref{e1})}{\meg} (\theta/4)|W|.
\end{eqnarray*}
This shows that part (a) is satisfied. For part (b), notice first that our assumption that $A$ contains no combinatorial line
of $[k+1]^n$ implies that $A\cap C\subseteq W\upharpoonright k$. Therefore, $\dens_W(A\cap C)\mik \lambda^{-m}\mik \lambda^{-M_0}\mik \eta$.
Since $\dens_W(A)\meg\delta-2\eta$ we see that $\dens_W\big(A\cap (W\setminus C)\big)\meg \delta-3\eta$. Moreover,
\[ \frac{\dens_W\big(A\cap(W\setminus C)\big)}{\dens_W(W\setminus C)}\meg
\frac{\delta-3\eta}{1-\theta/4} \meg(\delta-3\eta)(1+\theta/4) \stackrel{(\ref{e1})}{\meg} \delta+6\eta\]
and the proof is completed.
\end{proof}
The following corollary completes the first part of the proof of Proposition \ref{p6}. It shows that if $A$ contains no combinatorial
line, then it must correlate significantly with a ``structured" subset of $[k+1]^n$.
\begin{cor} \label{c11}
Let $0<\delta\mik 1$ and $m\in\nn$ with $m\meg M_0$. Let $n\meg n(\mathrm{GR}(k,m),\eta^2/2)$ and $A$ be subset of $[k+1]^n$ with
$\mathrm{dens}(A)\meg \delta$. Assume that $A$ contains no combinatorial line of $[k+1]^n$. Then there exist an $m$-dimensional subspace
$W$ of $[k+1]^n$ and a family $\{D_1,...,D_k\}$ of subsets of $W$ such that $D_i$ is $(i,k+1)$-insensitive in $W$ for every $i\in [k]$ and,
moreover, setting $D=D_1\cap ...\cap D_k$ we have $\mathrm{dens}_W(D)\meg \gamma$ and $\mathrm{dens}_W(A\cap D) \meg (\delta+\gamma)\mathrm{dens}_W(D)$.
\end{cor}
\begin{proof}
First assume that there exist an $m$-dimensional subspace $X$ of $[k+1]^n$ such that $\mathrm{dens}_X(A)\meg \delta +\eta^2/2$. Then we set
$W=X$ and $D_i=X$ for every $i\in[k]$. Since $\eta^2/2\meg \gamma$, it is clear that with these choices the result follows.
Otherwise, by Lemma \ref{l10}, there exist an $m$-dimensional subspace $W$ of $[k+1]^n$ and a set $C=C_1\cap ... \cap C_k$,
where $C_i$ is $(i,k+1)$-insensitive in $W$ for every $i\in [k]$, such that $\dens_W\big(A\cap(W\setminus C)\big)\meg(\delta+6\eta)
\dens_W(W\setminus C)$ and $\dens_W\big(A\cap(W\setminus C)\big)\meg \delta-3\eta$.

We set $P_1=W\setminus C_1$ and $P_i=(W\setminus C_i)\cap C_1\cap...\cap C_{i-1}$ if $i\in \{2,...,k\}$. Also, for every $i\in [k]$ let $\lambda_i=\mathrm{dens}_W (P_i)/\mathrm{dens}_W(W\setminus C)$ and $\delta_i=\mathrm{dens}_W(A\cap P_i)/\mathrm{dens}_W(P_i)$ with the
convention that $\delta_i=0$ if $P_i$ happens to be empty. The family $\{P_1,...,P_k\}$ is a partition of $W\setminus C$ and so
$\sum_{i=1}^k \lambda_i\delta_i =\dens_W\big(A\cap (W\setminus C)\big)/ \dens_W(W\setminus C)\meg \delta+6\eta$. Hence, there exists
$i_0\in [k]$ such that $\lambda_{i_0}\meg 3\eta/k$ and $\delta_{i_0}\meg \delta+3\eta$. We set $D_i=C_i$ if $i<i_0$,
$D_{i_0}=W\setminus C_{i_0}$ and $D_i=W$ if $i>i_0$. Clearly $D_i$ is $(i,k+1)$-insensitive in $W$ for every $i\in [k]$. Moreover,
we have $D_1\cap ...\cap D_k=P_{i_0}$ and so $\dens_W(P_{i_0})=\lambda_{i_0} \dens_W(W\setminus C) \meg (3\eta/k)(\delta-3\eta)\meg \gamma$
and $\dens_W( A\cap P_{i_0})=\delta_{i_0} \dens_W(P_{i_0}) \meg (\delta+3\eta)\dens_W(P_{i_0})\meg(\delta+\gamma)\dens_W(P_{i_0})$ as desired.
\end{proof}
The second part of the proof of Proposition \ref{p6} is a tilling procedure that enables us to partition any ``structured" subset of $[k+1]^n$
(that is, any subset of $[k+1]^n$ of the form $D_1\cap ... \cap D_k$ where $D_i$ is $(i,k+1)$-insensitive for every $i\in [k]$) in subspaces
of sufficiently large dimension. First one treats the case of insensitive sets. To this end, for every $0<\beta\mik 1$ and every integer
$m\meg 1$ we set
\begin{equation} \label{e5}
M_1=\mdhj^*(k,m,\beta) \ \ \text{ and } \ \ F(m,\beta)=\lceil \beta^{-1}(k+1+m)^{M_1}(k+1)^{M_1-m}M_1\rceil
\end{equation}
where $\mdhj^*(k,m,\beta)$ is as defined in Corollary \ref{c5}. We have the following lemma.
\begin{lem}\label{l12}
Let $0<\beta\mik 1$ and $m\in\nn$ with $m\meg 1$. Also let $i\in [k]$. If $n\meg F(m,\beta)$, then for every $(i,k+1)$-insensitive
subset $D$ of $[k+1]^n$ with $\dens(D)\meg2\beta$ there exists a family $\mathcal{V}$ of pairwise disjoint $m$-dimensional subspaces of
$[k+1]^n$ which are all contained in $D$ and are such that $\mathrm{dens}(D\setminus \cup\mathcal{V})<2\beta$.
\end{lem}
\begin{proof}
We set $\Theta=\beta (k+1+m)^{-M_1} (k+1)^{m-M_1}$. For every $x\in[k+1]^{n-M_1}$ let $D_x=\{y\in [k+1]^{M_1}: x^\con y\in D\}$.
Since $\ave_{x\in [k+1]^{n-{M_1}}}\dens(D_x)= \dens(D)\meg2\beta$, there exists a subset $T_1$ of $[k+1]^{n-M_1}$ with $\dens(T_1)\meg \beta$
such that $\dens(D_x)\meg\beta$ for every $x\in T_1$. Let $x\in T_1$ be arbitrary. By the choice of $M_1$ in (\ref{e5}) and Corollary \ref{c5},
there exists a subspace $V_x$ of $[k+1]^{M_1}$ of dimension $m$ such that $V_x\upharpoonright k\subseteq D_x$. It follows that
$x^\con (V_x\upharpoonright k)\subseteq D$, and so, $x^{\con}V_x\subseteq D$ since $D$ is $(i,k+1)$-insensitive. The number of
$m$-dimensional subspaces of $[k+1]^{M_1}$ is less than $(k+1+m)^{M_1}$. Therefore there exists a subspace $V_1$ of $[k+1]^{M_1}$
such that the set $S_1=\{x\in[k+1]^{n-M_1}: x^\con V_1\subseteq D\}$ has density at least $\beta(k+1+m)^{-M_1}$. Notice that $S_1$
is $(i,k+1)$-insensitive. We set $\mathcal{V}_1=\{x^\con V_1:x\in S_1\}$. It is clear that $\mathcal{V}_1$ is a family of pairwise disjoint
$m$-dimensional subspaces of $[k+1]^n$ such that $\cup \mathcal{V}_1\subseteq D$. Moreover, by the choice of $\Theta$, we have $\dens(\cup\mathcal{V}_1)\meg\Theta$.

If $\mathrm{dens}(D\setminus\cup\mathcal{V}_1)<2\beta$, then we are done. Otherwise let $D_1=D\setminus\cup\mathcal{V}_1$.
The set $D_1$ is not $(i,k+1)$-insensitive but is ``almost" insensitive in the following sense. For every $y\in [k+1]^{M_1}$
if we set $D_1^y=\{x\in [k+1]^{n-M_1}: x^{\con}y\in D_1\}$, then $D_1^y$ is $(i,k+1)$-insensitive. This is clear if $y\notin V_1$.
On the other hand if $y\in V_1$, then $D_1^y= \{x\in [k+1]^{n-M_1}: x^{\con}y\in D\}\setminus S_1$ and the claim follows since both
$D$ and $S_1$ are $(i,k+1)$-insensitive. Now for every pair $(x,y)\in [k+1]^{n-2M_1}\times [k+1]^{M_1}$ let
$D_1^{(x,y)}=\{z\in [k+1]^{M_1}: x^\con z^\con y\in D_1\}$. Using the previous remarks it is easily seen that the set $D_1^{(x,y)}$ is
$(i,k+1)$-insensitive. Moreover, $\ave_{(x,y)} \dens(D_1^{(x,y)})=\dens(D_1)\meg 2\beta$. Arguing precisely as before, it is possible
to select an $m$-dimensional subspace $V_2$ of $[k+1]^{M_1}$ such that the set $S_2=\{(x,y)\in [k+1]^{n-2M_1}\times [k+1]^{M_1}:
x^\con V_2^{\con} y\subseteq D_1\}$ has density at least $\beta(k+1+m)^{-M_1}$. Also observe that for every $y\in [k+1]^{M_1}$ the set
$S_2^y=\{x\in [k+1]^{n-2M_1}: x^{\con}V_2^{\con}y\in D_1\}$ is $(i,k+1)$-insensitive. We set $\mathcal{V}_2=
\mathcal{V}_1\cup\{ x^{\con}V_2^{\con}y: (x,y)\in S_2\}$. Then $\mathcal{V}_2$ is a new family of pairwise disjoint $m$-dimensional
subspaces of $[k+1]^n$ with $\cup\mathcal{V}_2\subseteq D$ and $\dens(\cup\mathcal{V}_2)\meg \dens(\cup\mathcal{V}_1) +\Theta$.

We continue similarly. At each step the density of the union of the members of the new collection of subspaces is increased by $\Theta$.
So this process must stop after at most $\lfloor \Theta^{-1} \rfloor$ iterations. Since $n\meg \beta^{-1}(k+1+m)^{M_1}(k+1)^{M_1-m}M_1= M_1/\Theta$
the above algorithm will eventually terminate and the proof is completed.
\end{proof}
By recursion on $r\in [k]$, for every $0<\beta\mik 1$ and every $m\in\nn$ with $m\meg 1$ we define the integer $F^{(r)}(m,\beta)$ by the rule
\begin{equation} \label{e6}
F^{(1)}(m,\beta)=F(m,\beta) \ \ \text{ and } \ \ F^{(r+1)}(m,\beta)=F^{(r)}\big(F(m,\beta),\beta\big).
\end{equation}
The following corollary completes the second part of the proof of Proposition \ref{p6}.
\begin{cor} \label{c13}
Let $0<\beta\mik 1$, $m\in\nn$ with $m\meg 1$ and $r\in [k]$. Let $n\meg F^{(r)}(m,\beta)$ and for every $i\in [r]$ let $D_i$ be an
$(i,k+1)$-insensitive subset of $[k+1]^n$. We set $D=D_1\cap ...\cap D_r$. If $\dens(D)\meg 2r\beta$, then there exists a family $\mathcal{V}$
of pairwise disjoint $m$-dimensional subspaces of $[k+1]^n$ which are all contained in $D$ and are such that
$\mathrm{dens}(D\setminus\cup\mathcal{V})<2r\beta$.
\end{cor}
\begin{proof}
By induction on $r$. The case ``$r=1$" follows from Lemma \ref{l12}. Assume that the result has been proved up to $r\in [k-1]$.
Fix $n\meg F^{(r+1)}(m,\beta)$ and let $D_1,...,D_{r+1}$ be a family of subsets of $[k+1]^n$ as described above. By our inductive hypothesis,
there exists a family $\mathcal{V}_1$ of pairwise disjoint $F(m,\beta)$-dimensional subspaces of $[k+1]^n$ which are all contained in
$D':=D_1\cap...\cap D_r$ and are such that $\dens(D'\setminus \cup\mathcal{V}_1)<2r\beta$. Let
$\mathcal{V}_2= \{ V\in\mathcal{V}_1: \dens_V(D_{r+1})\meg 2\beta\}$. For every $V\in\mathcal{V}_2$ let $\mathcal{B}_V$ be the collection
of $m$-dimensional subspaces of $V$ resulting by Lemma \ref{l12} when applied to the set $V\cap D_{r+1}$. We set
$\mathcal{V}=\{W: V\in\mathcal{V}_2 \text{ and } W\in\mathcal{B}_V\}$. Then $\mathcal{V}$ is as desired.
\end{proof}
We are now ready to give the proof of Proposition \ref{p6}.
\begin{proof}[Proof of Proposition \ref{p6}]
For every $0<\delta\mik 1$ and every $d\in\nn$ with $d\meg 1$ let $\beta=\gamma^2/4k$ and $m(d)=\max\{ M_0, F^{(k)}(d,\beta)\}$.
We define
\begin{equation} \label{e7}
N(k,d,\delta)=n\big(\mathrm{GR}\big(k,m(d)\big),\eta^2/2\big).
\end{equation}
Fix $n\meg N(k,d,\delta)$ and a subset $A$ of $[k+1]^n$ with $\dens(A)\meg\delta$. Assume that $A$ contains no combinatorial line of
$[k+1]^n$. By Corollary \ref{c11}, there exist a subspace $W$ of $[k+1]^n$ of dimension $m(d)$ and a family $\{D_1,...,D_k\}$ of subsets
of $W$ such that $D_i$ is $(i,k+1)$-insensitive in $W$ for every $i\in [k]$ and, setting $D=D_1\cap...\cap D_k$, we have
$\mathrm{dens}_W(D)\meg \gamma$ and $\mathrm{dens}_W(A\cap D)\meg(\delta+\gamma)\mathrm{dens}_W(D)$. By Corollary \ref{c13},
there exists a family $\mathcal{V}$ of pairwise disjoint $d$-dimensional subspaces such that $\cup \mathcal{V}\subseteq D$ and
$\dens_W(D\setminus\cup\mathcal{V})<2k\beta=\gamma^2/2$. Combining the previous estimates, we see that
$\dens_W(A\cap \cup \mathcal{V})\meg  (\delta+\gamma/2)\dens_W(\cup\mathcal{V})$. Hence, there exists $V\in\mathcal{V}$ such that
$\mathrm{dens}_W(A\cap V)\meg (\delta+\gamma/2)\mathrm{dens}_W(V)$ or equivalently $\mathrm{dens}_{V}(A)\meg \delta+\gamma/2$.
\end{proof}


\end{document}